\documentclass{article}
\usepackage[utf8]{inputenc}
\usepackage{microtype}
\usepackage{amsmath, amssymb, amsfonts, amsthm, mathtools}
\usepackage{MnSymbol}
\usepackage{graphicx}
\usepackage{tabularx}
\newtheorem{theorem}{Theorem}
\newtheorem{corollary}{Corollary}
\newtheorem{lemma}{Lemma}
\newtheorem{proposition}{Proposition}
\theoremstyle{definition}
\newtheorem{definition}{Definition}
\newtheorem{example}{Example}
\newtheorem{remark}{Remark}
\usepackage{float}
\usepackage{subfloat}
\usepackage{keyval2e}
\usepackage{everysel,ragged2e} 
\usepackage{caption} 
\usepackage{subfig}
\usepackage[labelfont={sf,bf},textfont=sf]{caption}
\usepackage{xcolor}   
\usepackage{hyperref} 
\usepackage{bookmark} 

\hypersetup{colorlinks=true,
            linkcolor={red!80!black},
            citecolor={blue!80!black},
            urlcolor={blue!80!black},
            pdftitle={Cretan(4t+1) Matrices},
            pdfauthor={N. A. Balonin and Jennifer Seberry},
            pdfsubject={Cretan(4t+1) Matrices},
              pdfkeywords={Hadamard matrices, } {regular Hadamard matrices, } {orthogonal matrices, } {Cretan matrices, }  {symmetric balanced incomplete block designs, } {SBIBD, } {05B20},
            }
\makeatletter
\DeclareCaptionTextFormat{Tablebookmark}{%
  \bookmark[rellevel=1,
            keeplevel,
            dest=\@currentHref,]{Table \thetable: #1}%
            #1%
}
\DeclareCaptionTextFormat{Figurebookmark}{%
  \bookmark[rellevel=1,
            keeplevel,
            dest=\@currentHref,]{Figure \thefigure: #1}%
            #1%
}
\makeatother
\newcommand{\footremember}[2]{%
\footnote{#2}
\newcounter{#1}
\setcounter{#1}{\value{footnote}}%
}

\author{N. A. Balonin\footremember{Nick} {Saint Petersburg State University of Aerospace Instrumentation,
67, B. Morskaia St., 190000, St. Petersburg, Russian Federation. Email: \url{korbendfs@mail.ru}}
and Jennifer Seberry\footremember{UoW} {School of Computing and Information Technology, Faculty of Engineering and Information Sciences, University of Wollongong, NSW 2522, Australia. Email: \url{jennifer_seberry@uow.edu.au}}
}
\title{Cretan(4t+1) Matrices}

\begin{document}

\maketitle

\begin{abstract}
A $Cretan(4t+1)$ matrix, of order $4t+1$, is an orthogonal matrix whose elements have moduli $\leq 1$.  The only $Cretan(4t+1)$ matrices previously published are for orders 5, 9, 13, 17 and 37. This paper gives infinitely many new $Cretan(4t+1)$ matrices constructed using $regular~Hadamard$ matrices, $SBIBD(4t+1,k,\lambda)$, weighing matrices, generalized Hadamard matrices and the Kronecker product. We introduce an inequality for the radius and give a construction for a Cretan matrix for every order $n \geq 3$.
\end{abstract}


\section{Introduction}\label{sec:introduction}

\indent An application in image processing (compression, masking) led to the search for orthogonal matrices, all of whose elements have modulus $\leq 1$ and which have maximal or high determinant.
  
$Cretan$ matrices were first discussed, per se, during a conference in Crete in 2014. This paper follows closely the joint work of N. A. Balonin, Jennifer Seberry and M. B. Sergeev \cite{BNSJ2014b,BNSJ2014a,BSS2015}.

The orders $4t$ (Hadamard), $4t-1$ (Mersenne), $4t-2$ (Weighing) are discussed in \cite{BM2006, BMS2012, BalSebSBIBD}. This present work emphasizes the $4t+1$ (Fermat type) orders with real elements $\leq 1$. Cretan matrices which are complex, based on the roots of unity or are just required to have at least one 1 are  mentioned.

\subsection{Preliminary Definitions}

The absolute value of the determinant of any matrix is not altered by 1) interchanging any two rows, 2) interchanging any two columns, and/or 3) multiplying any row/or column by $-1$. These equivalence operations are called \textit{Hadamard equivalence operations}. So the absolute value of the determinant of any matrix is not altered by the use of Hadamard equivalence operation.

Write $I_n$ for the identity matrix of order $n$, $J$ for the matrix of all 1's and let $\omega $ be a constant. An $orthogonal$ matrix, S, of order $n$, is square, has real entries and satisfies $SS^{\top }= \omega I_n$. The $core$ of a matrix is formed by removing the first row and column.

A $Cretan$ matrix, $S$, of order $n$ has entries with modulus $\leq 1$ and at least one 1 per row and column. It satisfies $SS^{\top } = \omega I_n$ and so it is an orthogonal matrix. A $Cretan(n; \tau; \omega)$ matrix, or $CM(n; \tau; \omega)$  has $\tau $ levels or values for its entries \cite{BNSJ2014b}. 

An \textit{Hadamard matrix} of order $n$ has entries $\pm 1$  and satisfies $HH^{\top} = nI_n$ for $n$ = 1, 2, $4t$, $t > 0$ an integer. Any Hadamard matrix can be put into \textit{normalized~ form}, that is having the first row and column all plus 1s using Hadamard equivalence operations: that is it can be written with a core.  A \textit{regular Hadamard matrix} of order $4m^2$ has $2m^2 \pm m$ elements 1 and $2m^2 \mp m$ elements $-1$ in each row and column (see \cite{WSW1972,SY92}).

Hadamard matrices and weighing matrices are well known orthogonal matrices. We refer to \cite{BNSJ2014a,JH1893,WSW1972,AGJS1979,SY92} for more details and other definitions. The reader is pointed to \cite{AB1962,DAD1979,WD1986a} for details of generalized Hadamard matrices, Butson Hadamard matrices and generalized weighing matrices.

For the purposes of this paper we will consider an $SBIBD(v, k, \lambda)$, $B$, to be a $v \times v$ matrix, with entries $0$ and $1$, $k$ ones per row and column, and the inner product of distinct pairs of rows and/or columns to be $\lambda$. This is called the \textit{incidence matrix} of the SBIBD. For these matrices $\lambda(v-1) = k(k-1)$, $BB^{\top } = (k - \lambda)I + \lambda J$, and det $B = k(k - \lambda)^{\frac{v-1}{2}}$.

For every $SBIBD(v, k, \lambda)$ there is a complementary $SBIBD(v, v-k, v-2k + \lambda)$. One can be made from the other by interchanging the $0$'s of one with the $1$'s of the other. The usual use $SBIBD$ convention that $v >2k$ and $k >2\lambda$ is followed.

We now define our important concepts the \textit{orthogonality equation}, the \textit{radius equation(s)}, the \textit{characteristic equation(s)} and the \textit{weight} of our matrices.

\begin{definition}[\textbf{Orthogonality equation, radius equation(s), characteristic equation(s), weight}]\label{def:or-rad-char}
Consider the matrix $S= (s_{ij})$ comprising  the variables $x_1,~x_2,~ \cdots ~x_{\tau}$.

The \textit{matrix orthogonality equation} 
\begin{equation}\label{orthogonality_equation}
 S^{\top }S = SS^{\top} = \omega I_n 
\end{equation} 
yields two types of equations: the $n$ equations which arise from taking the inner product of each row/column with itself (which leads to the diagonal elements of $\omega I_n$ being $\omega$) are called \textit{radius equation(s)}, $g(x_1,~x_2,~ \cdots ~x_{\tau})=\omega$, and the $n^2 -n$ equations, $f(x_1,~x_2,~ \cdots ~x_{\tau})=0$, which arise from taking inner products of distinct rows of $S$ (which leads to the zero off diagonal elements of $\omega I_n$  are called \textit{characteristic equation(s)}. Cretan matrices must satisfy the three equations: the orthogonality equation (\ref{orthogonality_equation}), the radius equation and the characteristic equation(s). \end{definition}

\noindent Notation: We use CM$(n;\tau;\omega;\textnormal{det} (optional);(t_1, t_2, \ldots , t_{\tau}))$,
or just $CM(n;\tau;\omega)$, where $t_1, t_2, \ldots , t_{\tau}$ are the possible values (or levels) of the elements in CM.

\subsection{Inequalities}\label{subsec:inequalities}

Some inequalities are known for matrices which have real entries $\leq 1$. Hadamard  matrices, $H= (H_{ij})$, which are orthogonal and with entries $\pm 1$ satisfy the equality of Hadamard's inequality (\ref{eq:hadamard}) \cite{JH1893}
\begin{equation} \label{eq:hadamard} \textnormal{det}(HH^{\top}) \leq \prod_{i=1}^n \sum_{j=1}^n |h_{ij}|^2,
\end{equation}
have determinant $\leq n^{\frac{n}{2}}$. 
Further Barba \cite{Barba33} showed that for matrices, $B$, of order $n$ whose entries are $\pm 1$, 
\begin{equation}\label{eq:B}
\det{B} \leq \sqrt{2n-1}(n-1)^{\frac{n-1}{2}} \textnormal{ or asymptotically } \approx 0.858(n)^{\frac{n}{2}}\,.
\end{equation}
\noindent For $n= 9$ Barba's inequality gives $det B \leq \sqrt{17} \times 8^4 = 16888.24$. The Hadamard inequality gives 19683 for the bound on the determinant of the $\pm 1$ matrix of order 9. So the Barba bound is better for odd orders. We thank Professor Christos Koukouvinos for pointing out to us that the literature, see Ehlich and Zeller, \cite{HEKZ1962}, yields a $\pm 1$ matrix of order 9  with determinant 14336. These bounds have not been met for $n$ = 9. 

Koukouvinos also pointed out that in Raghavarao \cite{DR1959} a $\pm 1$ matrix of order 13 with determinant 14929920 $\approx 1.49 \times {10}^7$ is given. This is the same value given for $n= 13$ given by Barba's inequality. The Hadamard inequality gives $1.74 \times {10}^7$ for the bound on the determinant of the $\pm 1$ matrix of order 13. 
 
These bounds have been significantly improved by Brent and Osborn \cite{RBJO2012a} to give $\leq(n+1)^{\frac{(n-1)}{2}}$.
               
Wojtas \cite{Wojtas64} showed that for matrices, $B$, whose entries are $\pm 1$, of order $n \equiv 2 \pmod{4} $ we have

\begin{equation}\label{eq:Cwotjas}  
\det{B} \leq 2(n-1)(n-2)^{\frac{n-2}{2}} \textnormal{ or asymptotically } \approx 0.736(n)^{\frac{n}{2}}\,.
\end{equation}
\noindent This gives a determinant bound $\leq $ 73728 for order 10 whereas the weighing matrix of order 10 has determinant $9^5 =59049$.

We observe that the determinant of a CM$(n;\tau; \omega; \textnormal{det})$ is always $\omega ^{\frac{n}{2}}$.

Hence we can rewrite the known inequalities of this subsection noting that only the Hadamard inequality applies generally for elements with modulus $\leq 1$. Thus we have

\begin{theorem} \label{had:ineq} {\bf Hadamard-Cretan Inequality}
The radius of a Cretan matrix of order $n$ is $\leq n.$
\end{theorem}

\section{Two Trivial \texorpdfstring{Cretan$(n)$}{Cretan(n)} Families}
The next two families are included for completeness.

\subsection{The Basic Family}

\begin{lemma} \label{lem;basic}
Consider $C=aJ+b(J-I)$ of order $n$, $a$, $b$ variables. This gives a $CM(n;2;1+\frac{4(n-1)}{(n-2)^2})$ matrix of order $n$ ie a $CM(n;2;1+\frac{4(n-1)}{(n-2)^2};\textnormal{det};(1,~\frac{-2}{n-2}))$.
\end{lemma}

\begin{proof} Writing $C$ with $a$ on the diagonal and other elements $b$, the radius and characteristic equations become
\[ a^2 + (n-1) b^2 = \omega  \textnormal{~~~and~~~} 2a + (n-2)b =0. \]
Hence with a =1 and b =$\frac{-2}{n-2}$ we have $\omega = 1 + \frac{4(n-1)}{(n-2)^2}$ for the required $CM(n)$ matrix. \end{proof}

\begin{remark} For $n$ = 7, 9, 11, 13 this gives $\omega $ = $1 \frac{24}{25}$, 
$1 \frac{32}{49}$, $1 \frac{40}{81}$ and $1 \frac{48}{121}$ respectively.
These determinants are very small. However they do give a $CM(n;2)$ for all integers $n > 0$. \end{remark}

\subsection{Known Families}

The following results may be found in \cite{BNMS2013a} and \cite{BalSebSBIBD}.

\begin{proposition}\label{prop1}{\bf [Cretan(4t)]}
There is a Cretan$(4t;2 ;4t)$ for every integer $4t$ for which there exists an Hadamard matrix.
\end{proposition}

\begin{proposition}\label{prop2}{\bf [Cretan(4t-1)]}
There are Cretan$(4t-1;2 ;\omega)$, $\omega = 4t  +1 -\sqrt{t}$ and $\omega = \frac{2t^3 + t -2t(2t-1)\sqrt{t}}{(t-1)^2}$ for every integer $4t$ for which there exists an Hadamard matrix.
\end{proposition}

The next two results are easy for the knowledgable reader and merely mentioned here.

\begin{proposition}\label{prop3}{\bf [Cretan(4t-2)]}
There are Cretan$(4t-2;3 ;k)$ whenever there is a $W(4t-2,k)$ weighing matrix. For $k=4t-3$, the sum of two squares, and a 4W(4t-2,4t-3) is known, the complex Cretan matrix CM(4t-2;3;4t-2) has elements $i = \sqrt{-1},~1$ or $-1$.
\end{proposition}

\begin{proposition}\label{prop4}{\bf [Cretan(np)]}
There are complex Cretan$(np;p ;n)$, whenever there exists a generalized Hadamard matrix based on the $p$th roots of unity.
\end{proposition}

\subsection{The Additive Families}

We will illustrate this construction using two Cretan matrices to give a Cretan matrix whose order is the sum of their orders. This shows how many
possible matrices we might find for any $n$ but again all the determinants are small.

\begin{lemma}
\label{lem:additive}
Let $A$ and $B$ be $CM(n_1;3;\omega_1)$ and $CM(n_2;3;\omega_2)$ respectively.
Then $A \oplus B$ given by 
\[\begin{bmatrix}
A & 0\\ 0 &  B
\end{bmatrix}\]
is a $CM(n_1 + n_2;4; \omega)$ matrix of order $n_1 + n_2$ with $\omega = \min(\omega_1,\omega_2)$. (Note it does not have one 1 per row and column.)
\end{lemma}

\begin{remark} 
We note using smaller $CM(n_i;\tau; \omega_i)$ gives many inequivalent $CM(n;\tau;\omega)$ for any order $n= \sum_{i}n_i$, but the elements of all but the smallest submatrix will not contribute 1 to the resulting  Cretan matrix. 

Now with $n= n_1+n_2$ for $21 = 4+17$, $5+16$, $6+15$, $7+14$, $8 +13$, $9+12$, $10+11$ plus other combinations, the submatrices of orders $n_1$ and $n_2$ contribute differently to $\tau$ and $\omega$. This means
\end{remark}

\begin{proposition}\label{all}
There is a Cretan$(n;\tau ;\omega)$ for every integer $n$.
\end{proposition}

In  Section \ref{kron} we explore the same Proposition \ref{all} for more interesting $\tau $.

\section{Constructions for Cretan\texorpdfstring{$(4t+1;\tau )$}{(4t+1;τ)} Matrices}
We now describe a number of constructions for Cretan$(4t+1)$ matrices.

\subsection{Constructions using SBIBD}

\subsubsection{2-level Cretan\texorpdfstring{$(4t+1)$}{(4t+1)} matrices via SBIBD\texorpdfstring{$(v=4t+1,k,\lambda)$}{SBIBD(v=4t+1,k,λ)}}

The following Theorem is a special case of the construction for 2-level Cretan($v=4t+1$) given in \cite{BalSebSBIBD}. It also yields a valid CM$(37;2)$.


\begin{theorem} \cite{BalSebSBIBD}\label{def:2-level-SBIBD-orthogonal-matrices}
Let $S$ be a CM$(v=4t+1;2;\omega;(a,b))$ based on $SBIBD(v=4t+1,k,\lambda)$ then $a=1$, ~$b=\frac{(k-\lambda)\pm \sqrt{k-\lambda}}{v-2k+\lambda}$ and ~$\omega = ka^2 +(v-k)b^2$, provided $|b|\leq 1$.
\end{theorem}

\begin{example}
Using the La Jolla Repository \url{http://www.ccrwest.org/ds.html} of difference sets that Marshall Hall Jr found an $SBIBD(37,9,2)$. Using Theorem \ref{BalSebSBIBD} we obtain
CM$(37;2;12.325;(1,0.345))$ and CM$(37;2;9.485;(1,0.132))$. The complementary $SBIBD(37,28,21)$
does not give any Cretan matrix as $|b|$ is $\geq 1$. 

We especially note the $(45,12,3)$ difference set, where the occurrence of the $Cretan(45,2;20\frac{1}{4})$ matrix and the $Cretan(45,2;14\frac{1}{16})$ matrices both arise from the $SBIBD(45,12,3)$: the complementary $SBIBD(45,33,24)$ does not yield any Cretan matrix.
\end{example}

\begin{example}\hfill
Orthogonal \hfill matrices \hfill of \hfill orders \hfill $13$ \hfill and \hfill $21$
\hfill may \hfill be \hfill constructed \hfill by \newline \hfill using \hfill the \hfill  $SBIBD(13,4,1)$ \hfill and \hfill $SBIBD(21,5,1)$  \hfill given \hfill in \hfill \cite{La-Jolla}). 
CM$(13;2;9.60;(1, \frac{3\pm \sqrt{3}}{6}))$  and CM$(21;2;10;(1, -\frac{1}{6}))$ are given in Figure \ref{fig:CM13-CM21}.
\begin{figure}[H]
  \centering
  \subfloat[][CM(13; 2; 9.60)]{\includegraphics[width=0.4\textwidth]{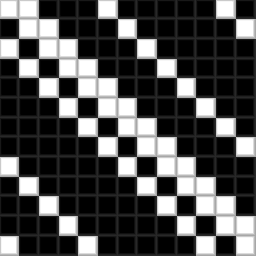}} \quad
  \subfloat[][CM(21; 2; 10)]{\includegraphics[width=0.4\textwidth]{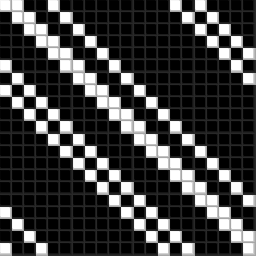}}
  \caption{2-level Cretan matrices of order 13 and 21}
  \label{fig:CM13-CM21}
\end{figure}
\end{example}

All the examples of $SBIBD(4t+1,k,\lambda)$ that we have given from the La Jolla Repository have been constructed using difference sets. Most of those we give arise from Singer difference sets and finite geometries: these $SBIBD((p^{n+1}-1)/(p-1), (p^{n}-1)/(p-1), (p^{n-1}-1)/(p-1))$ difference sets are denoted as $PG(n,p)$. The bi-quadratic type constructions are due to Marshall Hall \cite{MH1967a}. There are many $SBIBD$ constructed without using difference sets.

\subsubsection{Bordered Constructions}

We do not elaborate on the next theorem here but note it gives many Cretan matrices CM$(v+1)$.

\begin{theorem}\label{thm:bordered-SBIBD}
The matrix $C$ below can be used to construct many CM$(v+;\tau;\omega)$ with borders by replacing the matrix $B$ by an $SBIBD(v,k,\lambda)$.
\end{theorem}

When a matrix $C$ is written in the following form
\newcommand{\BigFig}[1]{\parbox{15pt}{\Huge #1}}
\newcommand{\BigB}{\BigFig{B}}
\[C = \begin{bmatrix}\label{bordered}
x      & s   & \dots & s  \\
s  \\
\vdots & & \BigB \\
s 
\end{bmatrix}\] 
$B$ is said to be the \textit{core of} $C$ and the $s $'s are the \textit{borders} of $B$ in $C$. $C$ is said to be in \textit{bordered form}. The variables $x$ are $s$ can be realized in the cases described below.

\subsubsection{Using Regular Hadamard Matrices}

For details and constructions many  of the known Regular Hadamard Matrices the interested reader is referred to \cite{SY92,WSW1972,XXS2003}.

\begin{lemma} \label{lem:RHMconstr} Let $M$ be a regular Hadamard matrix of order $4m^2$ with $2m^2+m$ positive elements per row and column. Then forming $C$ as follows

\renewcommand{\BigFig}[1]{\parbox{22pt}{\Huge #1}}
\newcommand{\BigM}{\BigFig{M}}
\[C = \begin{bmatrix}
1 & s & \dots & s  \\
s  \\
\vdots & & {\frac{1}{2m}} \BigM \\
s 
\end{bmatrix}\] gives a $Cretan(4n^2+1;4;1)$ matrix or CM$(4m^2+1;4;1;(0,1,\frac{1}{2m},\frac{-1}{2m}))$.
\end{lemma}

\begin{proof}
For $C$ to be a Cretan matrix it must satisfy the orthogonality, radius and characteristic equations. These are
\[ CC^{\top } = (1 + 4m^2s^2) I_{4m^2+1}= (s^2 +4m^2) I_{4m^2+1}= \omega I_{4m^2+1}\]
for the orthogonality equation, giving $s = 0$, $\omega = 1$ for the radius equation and
$0$ for the characteristic equations.

Hence we have a matrix of order $4m^2+1$ with elements 0, 1, $\pm \frac{1}{2m}$ satisfying the required Cretan equations. 
\end{proof} 

\begin{corollary} \label{pure-fermat}
Since there exists a regular (symmetric) Hadamard matrix of order $4=2^2$, $4^2=2^{2^2}$, $4^{4^2} = 2^{2^{2^2}}$ $\ldots $, there is a $Cretan(n=2^{2^{2^2}}$ $\ldots +1; 4;1)$ for $n$ a Fermat number.
\end{corollary}

\begin{proof} Let $S$ be the regular symmetric Hadamard matrix of order 4. Then the Kronecker product
\[ S \times S \times \ldots \times \]
is the required core for the construction in Lemma \ref{lem:RHMconstr}.
\end{proof}
\begin{example} Purported examples of pure Fermat matrices in Figure \ref{fig:F5-F17} for orders $5$ and $17$: levels $a,~b$ are white and black colours, the border level $s $ is given in grey. However the reader is cautioned that though the figures appear to be Cretan matrices they are not. They are based on SBIBD, including the regular Hadamard matrix SBIBD$(4m^2,2m^\pm m, m^\pm m)$ and require $c=a$. We note though that when $c=a\neq 1$ the radius and characteristic equations do not give meaningful real solutions.

\begin{figure}[h]
  \centering
  \subfloat[a][F(\texorpdfstring{$n=5$}{n=5})]{\includegraphics[width=0.3\textwidth]{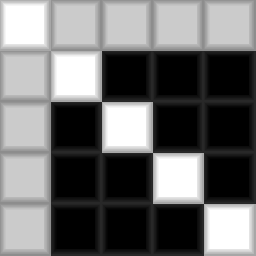}} \quad
  \subfloat[b][F(\texorpdfstring{$n=17$}{n=17})]{\includegraphics[width=0.3\textwidth]{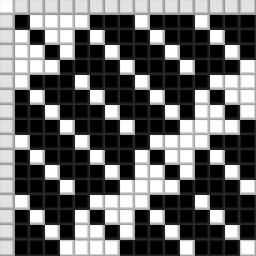}}
  \caption{Orthogonal Cretan(Fermat) matrices for Fermat numbers  5 and 17}
  \label{fig:F5-F17}
\end{figure}
\end{example}

\begin{example} See Figure \ref{fig:H36-CM37} for examples of a regular Hadamard matrix of order $36$ and a purported new  Balonin-Seberry type of 3-level Cretan($37$) with complex entries that is a orthogonal matrix of order $37$. A real Cretan($37;2)$ does exist from Theorem \ref{def:2-level-SBIBD-orthogonal-matrices} above (see example).

\begin{figure}[h]
  \centering
  \subfloat[a][H(\texorpdfstring{$n=36$}{n=36})]{\includegraphics[width=0.4\textwidth]{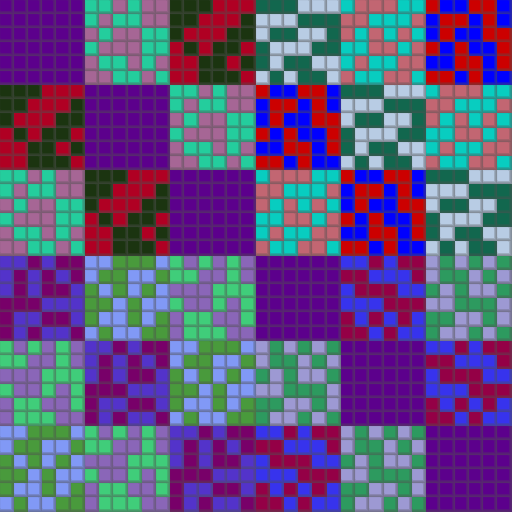}} \quad
  \subfloat[b][CM(\texorpdfstring{$n=37$}{n=37})]{\includegraphics[width=0.4\textwidth]{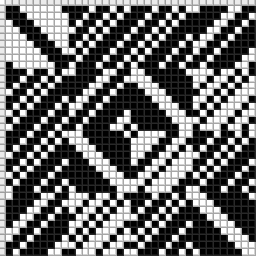}}\\
  \caption{Regular Hadamard matrix of order 36 and a 3-level Cretan(37)}
  \label{fig:H36-CM37}
\end{figure}

\end{example} \qed

\subsection{Using Normalized Weighing Matrix Cores } \label{subsec:normal-weighing-core-construct}
Thie next construction is not valid in the real numbers. However we can allow Cretan matrices to have complex elements and choose the diagonal to be $i = \sqrt{-1}$. 

\begin{lemma} \label{W(n,n-1)}
Suppose there exists a normalized conference matrix, $B$, of order $4t+2$, that is a $W(4t+2,4t+1)$. Then
$B$ may be written as
\renewcommand{\BigFig}[1]{\parbox{16pt}{\Huge #1}}
\newcommand{\BigF}{\BigFig{F}}
\[B = \begin{bmatrix}
i & 1 & \dots & 1\\
1 \\
\vdots  & & \BigF\\
1
\end{bmatrix}.\]

This is a Cretan matrix. \end{lemma}

Removing the first row and column of $B$ to study the core $F$ is unproductive.

\subsubsection{Generalized Hadamard Matrices and Generalized Weighing Matrices}

We first note that the matrices we study here have elements from groups, abelian and non-abelian, (see \cite{AB1962,DAD1979,WD1986a,JRS1979b,XSAWWX08} for more information) and may be written in additive or multiplicative notation. The matrices may have real elements, elements $\in \{1,-1\}$, elements $|n| \leq 1$, elements $\in \{1,i,~i^2=-1\}$, elements $\in \{1,i,-1,-i,~i^2=-1\}$, integer elements $\in \{a+ib,~i^2=-1\}$, $n$th roots of unity, the quaternions \{1 and $i,j,k,~i^2=j^2=k^2 =-1,ijk=-1\}$, $(a+ib) + j(c+id)$, \textit{a, b, c, d, integer} and \textit{i, j, k} quaternions or otherwise as specified.

We use the notations $B^{\top}$ for the transpose of $G$, $B^H$ for the group transpose, $B^C$ for the complwx conjugate of $B^{\top}$, $B^Q$ for the quaternion conjugate and $B^V$ for the quaternion conjugate transpose.

In all of these matrices the inner product of distinct rows a and b is a - b or $a.b^{-1}$ depending on whether the group is written in additive or multiplicative form.

\begin{itemize}

\item \textbf{Generalized orthogonality:} A \textit{generalized Hadamard matrix}, or \textit{difference matrix}, $GH(gn,g)$ over a group of order $g$ has the inner product of distinct rows the whole group the same number of $n$ times. The inner product is $\{g_{i1}g_{j1}^{-1},\dots,g_{in}g_{jn}^{-1}\}$. For example
\[G = \begin{bmatrix}
1 & 1 & 1 & 1\\ 1 & a & b & ab\\ 1 & b  & ab & a \\ 1 & ab & a & b 
\end{bmatrix}; \qquad GG^H = (group)I_4 = (Z_2 \times Z_2)I\]
orthogonality is because of the definition of the inner product.

\item \textbf{Butson Hadamard matrix} \cite{AB1962}
\[B= \begin{bmatrix} 1 & 1  & 1\\ 1& \omega & \omega^2 \\ 1 & \omega^2 & \omega
\end{bmatrix}; \quad BB^C =3I_3, \quad w^3 = 1, \quad 1+w+w^2 = 0\]
is said to be a Butson Hadamard matrix. Orthogonality depends on the fact that the $n$ $n$th roots of unity add to zero.

\item A \textit{generalized/generalized Hadamard matrix} \cite{AB1962,DAD1979,WD1986a}, $GH(np,G)$, where $G$ is a group of order $p$, can also be written in additive form for example:
\[ \begin{bmatrix}
0& 0& 0& 0& 0& 0\\
0& 0& 1& 2& 2& 1\\
0& 1& 0& 1& 2& 2\\
0& 2& 1& 0& 1& 2\\
0& 2& 2& 1& 0& 1\\
0& 1& 2& 2& 1& 0
\end{bmatrix} \quad \text{is a } GH(6,Z_3). 
\]
\item A \textit{generalized weighing matrix},$ W=GW(np,G,k)$ \cite{JRS1979b},where $G$ is a group of order $p$, has $w$ non-zero elements in each column and $W$ is orthogonal over $G$.  
The following two matrices are additive and multiplicative $GW(5,Z_3)$, respectively. 
\[ \begin{tabular}{ccc}
$ \begin{bmatrix}
*& 0& 0& 0& 0\\
0& *& 1& 2& 0\\
0& 1& *& 0& 2\\
0& 2& 0& *& 1\\
0& 0& 2& 1& * 
\end{bmatrix} 
$&
$\begin{bmatrix}
0&   1&   1&   1& 1\\
1&   0&   w& w^2& 1\\
1&   w&   0&   1& w^2\\
1& w^2&   1&   0& w\\
1&   1& w^2&   w& 0 
\end{bmatrix}.$ 

 \end{tabular} \]
\noindent{* is zero but not the zero of the group.}
\end{itemize}
\begin{theorem}
Any generalized Hadamard matrix or generalized weighing matrix is a CM$(n;g)$ over the group $G$, of order $g$, which may be the roots of unity.
\end{theorem}
\subsection{The Kronecker Product of Cretan Matrices} \label{kron}

\begin{lemma}
Suppose $A$ and $B$ are $CM(n_1;\tau_1;\omega_1)$ and $CM(n_2;\tau_2;\omega_2)$ then the Kronecker product of $A$ and $B$ written $A \times B$ is a $CM(n_1n_2;\tau ;\omega_1 \omega_2)$ where $\tau$ depends on $\tau_1$ and $\tau_2$.
\end{lemma}

\begin{example} From \cite{BalSebSBIBD,BHSS2015} we see that $CM(3;2;2.25)$, $CM(7;2;5.03)$ and $CM(7;2;3.34)$ exist so there exist $CM(21;3;11.32)$ and $CM(21;3;7.52)$.

The Hadamard-Cretan bound gives, for $n=21$, radius $\leq 21$.
\end{example}
From Balonin and Seberry \cite{BalSebSBIBD} we have that since an SBIBD$(p^r, \frac{p^r -1}{2},\frac{p^r -3}{4})$ exists for all prime powers $p^r \equiv 3\pmod{4}$ there exist CM$(p^r; 2; \omega)$ for all  these prime powers (see Proposition \ref{prop2}).
Hence using Kronecker products in the previous theorem and writing $n$ as a product of prime powers we have

\begin{theorem}
\label{thm:all-orders}
There exists a CM$(n;\tau;\omega)$ $\omega > 1$ for all odd orders $n$, $n = \prod \rho \times p^{i_1}p^{i_2}...$,  where $\rho$ is an order for which a Cretan
CM($\rho =4t+1$) is known and $p^{i_1}$, $p^{i_2}, \cdots $ are any prime powers $\equiv 3 \pmod{4}$, for some $\tau $ and $\omega$.
\end{theorem}
table
Table \ref{table:summary} gives the present integers for which $\rho $ is known. Similar theorems can be obtained for all even $n$.

\begin{remark} We note that $\tau $ depends on the actual construction used. Combining CM$(n_1;2;\omega_1:(a,b))$ and CM$(n_2;2;\omega_1:(a,b))$ gives CM$(n_1n_2;3;\omega_{12}:(a^2,ab,b^2))$. General formulae for $\tau$ from CM with different levels are left as an exercise.
\end{remark}

\section{The Difference between Cretan\texorpdfstring{($4t+1;\tau$)}{(4t+1;τ)} Matrices and Fermat Matrices}\label{sec:conclusion}

The first few pure Fermat numbers are  $v = 3, 5, 17, 257, 65537, 4294967297, \ldots$. We note these are all $ \equiv 1 \pmod{4}$ and may be constructed using Corollary \ref{pure-fermat}. Figure \ref{fig:fermat-matrix} gives an early example of a Fermat matrix. 

\begin{figure}[h]
  \centering
  \includegraphics[width=0.35\textwidth]{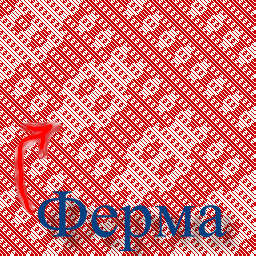}  
    \caption{Core of Russian Fermat Matrix from mathscinet.ru}
     \label{fig:fermat-matrix}
\end{figure} 

Finding 3-level orthogonal matrices of order $ \equiv 1 \pmod{4}$ for non-pure Fermat numbers has proved challenging. Orders $n=9$ and $n=13$ are given in cite{>}.

Orders $v=2^{even} + 1$ called Fermat type matrices, pose an interesting class to study.  

Orders $4t+1$, $t$ is odd, are $Cretan(4t+1)-$ matrices; their order is neither a Fermat number ($2+1 = 3$, $2^{2}+1=4+1$, $\;2^{2^{2}}+1=16+1$, $\;2^{2^{2^{ 2}}}+1=256+1$, $\ldots$) nor a Fermat type number ($2^{even}+1$). Examples of regular Hadamard matrices of order 36, giving the first $CM(37;3;1)$ matrix of order 37 \cite{BSS2015} where 37 is not a Fermat number or Fermat type number, have been placed at site \cite{JSsite36}. They use regular Hadamard matrices as a core and have the same, as any other Hadamard matrix, level functions. We call them $Cretan(4t+1)$ matrices and will consider them further in  our future work.

Matrices of the $Cretan(4t+1)$ family made from Singer difference sets (see \cite{MH1967a} also have orders belonging to the set of numbers $4t+1$, $t$ odd: these are different from the three-level matrices of Balonin-Sergeev (Fermat) family \cite{BNSM2014,BNMS2013a} with orders $4t+1$, $t$ is 1 or even.

\section{Summary}\label{sec:summary}
In this paper we have given new constructions for CM$(4t+1)$. These are summarised in Table \ref{table:summary} for $4t+1 < 200$.

\begin{table}[H]\centering
\caption{Some Cretan \texorpdfstring{$CM(4t+1),~3\leq 4t+1 \leq 199$}{CM(4t+1), 3 ≤ 4t+1 ≤ 199}}
\label{table:summary}
\begin{tabular}{l|rrrrr}
\hline \\[-1.75ex]
From Regular Hadamard Matrices ($\omega = 1$)& 5   & 17  & 37  & 45 & 65\\
                                        & 101 & 145 & 197 &    &   \\ 
\hline
\end{tabular}
\\
\vspace{1ex}
\begin{tabular}{rrr|lc|l}
\hline \\[-1.99ex]
\multicolumn{6}{l}{From Difference Sets (ds)}\\
$v$ & $k$& $\lambda$& Existence& Difference set& Comment \\ \hline \\[-2.5ex]
13  & 4  & 1        & All Known& PG(2,3)       & Unique Hall \cite{MH1956}\\
21  & 5  & 1        & All Known& PG(2,4)       & Unique Hall \cite{MH1956}\\
37  & 9  & 2        & Exists   & Biquadratic residue ds & Hall \cite{MH1956}\\
45  & 12 & 3        & All Known&               & La Jolla \cite{La-Jolla}\\
57  & 8  & 1        & All Known& PG(2,7)       &    Unique Hall \cite{MH1956}\\
73  & 9  & 1        & All Known& PG(2,8)       &    Unique Hall \cite{MH1956}\\
85  & 21 & 5        &    Exists& PG(3,4        & \cite{La-Jolla}\\
101 & 25 &  6       &    Exists& Biquadratic residue ds& Hall \cite{MH1956}\\
109 & 28 &  7       &    Exists& Biquadratic residue ds& Hall \cite{MH1956}\\
121 & 40 & 13       &    Exists& PG(4,3)       & \cite{La-Jolla}\\
133 & 33 &  8       &    Exists&               & La Jolla \cite{La-Jolla}\\
197 & 49 & 12       &    Exists& Biquadratic residue ds & Hall \cite{MH1956}\\ \\[-2.5ex]
\hline
\end{tabular}
\\
\vspace{1ex}
\begin{tabular}{l|l}
\hline \\[-1.99ex]
Kronecker Product & All Orders which are the Product a Known Order\\
                  & and of Prime Power $\equiv 3 \pmod{4}$\\
\hline
\end{tabular}
\end{table}

\begin{table}[H]\centering
\caption{Cretan \texorpdfstring{2-level and 3-level $CM(4t\pm 1),~3\leq 4t+1~\leq 199$}{CM(4t ±1), 3 ≤ 4t+1 ≤ 199}}
\label{table:summary3}
\begin{tabular}{|c|c||c|c||c|c|}
\hline \\[-1.99ex]
$v$ & $Method$ & $v$ & $Method$ &$v$ & $Method$ \\
\hline 
 3 & BM\cite{BM2006}+Prop:\ref{prop2} &  5 & BM\cite{BM2006} & 7 & BM+Prop:\ref{prop2} \\
 9 & BM\cite{BM2006}                  & 11 & BM\cite{BM2006}+Prop:\ref{prop2} & 13 & BM\cite{BM2006} \\
15 & Kronecker                 &  17 & & 19 & Prop:\ref{prop2} \\
21 & from SBIBD Table:\ref{table:summary}& 23 & Prop:\ref{prop2} & 25 & Kronecker \\
27 & Prop:\ref{prop2} & 29 & & 31 & Prop:\ref{prop2} \\
33 & Kronecker & 35 & Kronecker        & 37 & \\
39 & Kronecker       & 41 & & 43 & Prop:\ref{prop2} \\
45 & from SBIBD Table:\ref{table:summary} & 47 & Prop:\ref{prop2} & 49 & Kronecker \\
51 &                 & 53 & & 55 & Kronecker    \\          
57 & from SBIBD Table:\ref{table:summary} & 59 &  Prop:\ref{prop2}   & 61 & \\
63 & Kronecker       & 65 & Kronecker       & 67 & Prop:\ref{prop2} \\
69 & Kronecker & 71 & Prop:\ref{prop2} & 73 & from SBIBD Table:\ref{table:summary}\\
75 & Kronecker        & 77 & Kronecker & 79 & Prop:\ref{prop2}\\
81 & Prop:\ref{prop2}& 83 & & 85 & from SBIBD Table:\ref{table:summary}\\
87 & & 89 & & 91 & Kronecker\\
93 & Kronecker & 95 & Kronecker & 97 &\\
99 & Kronecker        & 101 & from SBIBD Table:\ref{table:summary} & 103& Prop:\ref{prop2}\\
105& Kronecker & 107 & Prop:\ref{prop2} & 109 & from SBIBD Table:\ref{table:summary}\\
111& & 113 & & 115 & Kronecker\\
117& Kronecker & 119 & & 121& from SBIBD Table:\ref{table:summary}\\
123& & 125& Kronecker & 127 & Prop:\ref{prop2}\\
129& Kronecker & 131 & Prop:\ref{prop2} & 133& from SBIBD Table:\ref{table:summary}\\
135& Kronecker & 137 & & 139 & Prop:\ref{prop2}\\
141& Kronecker & 143 & & 145& \\
147& Kronecker & 149 & & 151& Prop:\ref{prop2}\\
153& & 155 & Kronecker & 157& \\
159& & 161 & Kronecker & 163 & Prop:\ref{prop2}\\
165& Kronecker & 167 & Prop:\ref{prop2} & 169 & Kronecker\\
171& Prop:\ref{prop2} & 173 & & 175 & Kronecker\\
177& Kronecker & 179 & Prop:\ref{prop2} & 181 & \\
183& & 185 & & 187 & \\
189& Kronecker & 191 & Prop:\ref{prop2} & 193 & \\
195& Prop:\ref{prop2} & 197 & from SBIBD Table:\ref{table:summary}&  199& Prop:\ref{prop2}\\
\hline 
\end{tabular}
\end{table}

\section{Conclusions}

Cretan matrices are a very new area of study. They have many research lines open: what is the minimum number of variables that can be used; what are the determinants and radii that can be found for Cretan($n;\tau$) matrices; why do the congruence classes of the orders make such a difference to the proliferation of Cretan matrices for a given order; find the Cretan matrix with maximum and minimum determinant for a given order; can one be found with fewer levels?

\section{Acknowledgements}

 We thank Professors Richard Brent, Christos Koukouvinos and Ilias Kotsireas for their valuable input to this paper. The authors also wish to sincerely thank Mr Max Norden, BBMgt(C.S.U.), for his work preparing the layout and LaTeX version of this article.
We acknowledge \url{ http://www.wolframalpha.com} for the number calculations in this paper and \url{http://www.mathscinet.ru} for the graphics.

%

%
%
\bibliographystyle{unsrt}
\bibliography{Cretan-refs-1,Cretan-refs-2A}

\begin{thebibliography}{10}

\bibitem{BNSJ2014b}
N.~A. Balonin and Jennifer Seberry.
\newblock Remarks on extremal and maximum determinant matrices with moduli of
  real entries $\leq 1$.
\newblock {\em Informatsionno-upravliaiushchie sistemy}, 71(5):2--4, 2014.

\bibitem{BNSJ2014a}
N.~A. Balonin and Jennifer Seberry.
\newblock A review and new symmetric conference matrices.
\newblock {\em Informatsionno-upravliaiushchie sistemy}, 71(4):2--7, 2014.

\bibitem{BSS2015}
Jennifer~Seberry N.~A.~Balonin and M.B. Sergeev.
\newblock Three level {Cretan} matrices of order 37.
\newblock {\em Informatsionno-upravliaiushchie sistemy}, 2(2):2--3, 2015.

\bibitem{BM2006}
N.~A. Balonin and L.~A. Mironovski.
\newblock Hadamard matrices of odd order.
\newblock {\em Informatsionno-upravliaiushchie sistemy}, 3:46--50, 2006.
\newblock In Russian.

\bibitem{BMS2012}
N.A. Balonin, L.A. Mironovski, and M.B. Sergeev.
\newblock Calculation of {Hadamard-Fermat} matrices.
\newblock {\em Informatsionno-upravliaiushchie sistemy [Information and Control
  Systems]}, 6:90--93, 2012.
\newblock (In Russian).

\bibitem{BalSebSBIBD}
N.~A. Balonin and Jennifer Seberry.
\newblock Two-level {Cretan} matrices constructed using {SBIBD}.
\newblock (accepted 25 July 2015, Special Matrices), 2015.

\bibitem{WSW1972}
Jennifer~Seberry Wallis.
\newblock Hadamard matrices.
\newblock In {\em Combinatorics: Room {Squares}, {Sum-free} {Sets}, {Hadamard}
  {Matrices}}, volume 292 of {\em Lecture Notes in Mathematics}, page 508.
  Springer-Verlag, 1972.

\bibitem{SY92}
Jennifer Seberry and Mieko Yamada.
\newblock Hadamard matrices, sequences and block designs.
\newblock In J.~H. Dinitz and D.~R. Stinson, editors, {\em Contemporary Design
  Theory - a Collection of Surveys}, pages 431--560. Wiley, New York, 1992.

\bibitem{JH1893}
J.~Hadamard.
\newblock R\'esolution d'une question relative aux d\'eterminants.
\newblock {\em Bulletin des Sciences Mathematiques}, 17:240--246, 1893.

\bibitem{AGJS1979}
A.~V. Geramita and J.~Seberry.
\newblock {\em Orthogonal Designs: {Quadratic} Forms and {Hadamard} Matrices}.
\newblock Marcel Dekker, 1979.

\bibitem{AB1962}
A.~T. Butson.
\newblock Generalized {Hadamard} matrices.
\newblock {\em Proc. Amer. Math. Soc.}, 3:894--898, 1962.

\bibitem{DAD1979}
D.A. Drake.
\newblock Partial $\lambda$-geometries and generalised hadamard matrices.
\newblock {\em Canadian J. Math.}, 31:617--627, 1979.

\bibitem{WD1986a}
Warwick de~Launey.
\newblock A survey of generalised {Hadamard} matrices and difference matrices
  $d(k, \lambda ; g)$ with large $k$.
\newblock {\em Utilitas Math.}, 30:5--29, 1986.

\bibitem{Barba33}
G.~Barba.
\newblock Intorno al teorema di {Hadamard} sui determinanti a valore massimo.
\newblock {\em Giornale di Matematiche di Battaglini}, 71:70--86, 1933.

\bibitem{HEKZ1962}
H.~Ehlich and K.~Zeller.
\newblock Bin\"{a}re {Matrizen}.
\newblock {\em Z. Angew. Math. Mech.}, 42:T20--T21, 1962.

\bibitem{DR1959}
D.~Raghavarao.
\newblock Some optimum weighing designs.
\newblock {\em Ann. Math. Stat.}, 30:295--303, 1959.

\bibitem{RBJO2012a}
Richard~P. Brent and Judy-anne~H. Osborn.
\newblock On minors of maximal determinant matrices.
\newblock {\em ArXiv e-prints}, 08 2012.
\newblock Also published: Journal of Integer Sequences 16 (2013), Article
  13.4.2, 30 pp.

\bibitem{Wojtas64}
M.~Wojtas.
\newblock On {Hadamard's} inequality for the determinants of order
  non-divisible by 4.
\newblock {\em Colloquium Mathematicum}, 12:73--83, 1964.

\bibitem{BNMS2013a}
N.~A. Balonin and M.~B. Sergeev.
\newblock On the issue of existence of {Hadamard} and {Mersenne} matrices.
\newblock {\em Informatsionno-upravliaiushchie sistemy}, 66(2):89--90, 2013.
\newblock In Russian.

\bibitem{La-Jolla}
La {J}olla {D}ifference {S}et {R}epository.
\newblock Online: \url{www.ccrwest.org/ds.html}.
\newblock Viewed 2014:10:03.

\bibitem{MH1967a}
M.~{Hall, Jr}.
\newblock {\em Combinatorial Theory}.
\newblock Blaisdell Ginn and Co., 1 edition, 1967.

\bibitem{XXS2003}
Tianbing Xia, Mingyuan Xia, and Seberry Jennifer.
\newblock Regular {Hadamard} matrix, maximum excess and {SBIBD}.
\newblock {\em AJC}, 27:263--275, 2003.

\bibitem{JRS1979b}
Jennifer Seberry.
\newblock Some remarks on generalized {Hadamard} matrices and theorems of
  {Rajkundlia} on {SBIBD}s.
\newblock In A.F. Horadam and W.D. Wallis, editors, {\em Combinatorial
  Mathematics VI}, volume 748 of {\em Lecture Notes in Mathematics}, pages
  154--164. Springer-Verlag, Berlin-Heidelberg-New York, 1979.

\bibitem{XSAWWX08}
Sarah Spence Adams Beata Wysocki Tadeusz~Wysocki Jennifer~Seberry,
  Ken~Finlayson and Tianbing Xia.
\newblock The theory of quaternion orthogonal designs.
\newblock {\em IEEE Transactions on Signal Processing}, 56 1:256--265, 2008.

\bibitem{BHSS2015}
Jennifer~Seberry N.~A.~Balonin, Ofer~Hadar and M.B. Sergeev.
\newblock Three level {Cretan} matrices constructed via conference matrices.
\newblock {\em Informatsionno-upravliaiushchie sistemy}, 74(2):4--6, 2015.

\bibitem{JSsite36}
Jennifer Seberry.
\newblock Regular {Hadamard} matrices of order 36.
\newblock \url{http://www.uow.edu.au/~jennie/matrices/H36/36R.html}.

\bibitem{BNSM2014}
N.~A. Balonin and M.~B Sergeev.
\newblock Local maximum determinant matrices.
\newblock {\em Informatsionno-upravliaiushchie sistemy [Information and Control
  Systems]}, 68(1):2--15, 2014.
\newblock In Russian.

\bibitem{MH1956}
M.~{Hall, Jr}.
\newblock A survey of difference sets.
\newblock {\em Proc. Amer. Math. Soc.}, 7:975--986, 1956.

\end{thebibliography}
\end{document}